\documentclass{amsart}

\usepackage{layout}
\usepackage{amsmath, amsthm}
\usepackage{amssymb, amsrefs}
\usepackage{stmaryrd}
\usepackage{graphicx}
\usepackage{setspace}
\usepackage{textcomp}
\usepackage{esvect}
\usepackage[letterpaper]{geometry}
\usepackage{arydshln}
\usepackage[all]{xy}
\usepackage{etex}

\usepackage[dvipsnames]{xcolor}

\usepackage{epsfig}
\usepackage{latexsym}
\usepackage{amsthm}

\usepackage{mathrsfs}

\usepackage[colorlinks,citecolor=blue]{hyperref}

\usepackage[latin1]{inputenc}

\usepackage{tikz-cd}
\usetikzlibrary{graphs,graphs.standard}
\usepackage{pgfplots}
\usetikzlibrary{matrix,arrows,decorations.pathmorphing}
\usetikzlibrary{shapes.geometric}

\newtheorem{theorem}{Theorem}[section]
\newtheorem{lemma}[theorem]{Lemma}
\newtheorem{proposition}[theorem]{Proposition}
\newtheorem{prop}[theorem]{Proposition}
\newtheorem{corollary}[theorem]{Corollary}

\theoremstyle{definition}
\newtheorem{definition}[theorem]{Definition}

\newtheorem{example}[theorem]{Example}

\newtheorem{obs}[theorem]{Observation}

\theoremstyle{definition}
\numberwithin{equation}{section}

\tikzset{
    ncbar angle/.initial=90,
    ncbar/.style={
        to path=(\tikztostart)
        -- ($(\tikztostart)!#1!\pgfkeysvalueof{/tikz/ncbar angle}:(\tikztotarget)$)
        -- ($(\tikztotarget)!($(\tikztostart)!#1!\pgfkeysvalueof{/tikz/ncbar angle}:(\tikztotarget)$)!\pgfkeysvalueof{/tikz/ncbar angle}:(\tikztostart)$)
        -- (\tikztotarget)
    },
    ncbar/.default=0.5cm,
}

\tikzset{square left brace/.style={ncbar=0.5cm}}
\tikzset{square right brace/.style={ncbar=-0.5cm}}

\tikzset{round left paren/.style={ncbar=0.5cm,out=120,in=-120}}
\tikzset{round right paren/.style={ncbar=0.5cm,out=60,in=-60}}

\DeclareMathOperator{\Hom}{Hom}

\def \PX {\mathfrak{W}G}
\def \W {\mathfrak{W}}

\def \Gph {\mathsf{Gph}}

\def \con {\sim}
\definecolor{laura}{rgb}{.4, 0, .6}

\begin{document}
\pagestyle{headings}

\providecommand{\keywords}[1]
{
  \small	
  \textbf{\textit{Keywords---}} #1
}

 \title{Fundamental Groupoids for Graphs}
    \author{T. Chih and L. Scull}

\begin{abstract}
    In this paper, we develop a $\times$-homotopy fundamental groupoid for graphs, and show a functorial relationship to the 2-category of graphs.  We further explore the fundamental groupoid of graph products and develop a groupoid product which respects the graph product.  A van Kampen Theorem for these groupoids is provided.  Finally, we generalize previous work on a fundamental group for graphs, developing a  looped walk groupoid and showing a connection to  the polyhedral complex of graph morphisms.

\end{abstract}


\maketitle


\section{Introduction}

The fundamental group and its associated groupoid are instrumental in the study of the structure of topological spaces, capturing information about their homotopy type.  Thus when we look at the homotopy type of graphs it is natural to develop analogous structures.

There are several different definitions of homotopy for graphs.  Two of particular prominence are $A$-homotopy \cite{Babson1, Barcelo1, BoxHomotopy, hardeman2019lifting} and $\times$-homotopy \cite{Docht1, Docht2, CompGH, Kosolov1, Kosolov2, Kosolov3, KosolovShort}.  Previous work developing fundamental groups for graphs include an $A$-fundamental group \cite{Babson1} and a $\times$-homotopy fundamental group for exponential graphs   \cite{Docht2}.  

In this paper, we focus on $\times$-homotopy and develop a fundamental groupoid for any graph.   We build on previous work in \cite{CS1} developing the  homotopy category for $\times$-homotopy, and analyzing the structure of $\times$-homotopies from finite graphs via decomposition into 'spider moves'.   Using this as a foundation, we define a fundamental groupoid for graphs, and explore some of its properties.  We also discuss a looped version of our groupoid which generalizes the fundamental group of Dochtermann \cite{Docht2}.

Our paper is structured as follows.  Section \ref{S:Back} contains background results.       Section \ref{S:WalkG}, describes a groupoid of walks in $G$, defining a  functor from $\Gph$ to $\mathsf{Groupoids}$.  In Section \ref{S:fgp}, we define our fundamental groupoid as a quotient of the walk groupoid, where arrows are homotopy classes of walks.  We show that this defines a functor from the homotopy category of graphs to groupoids, giving a homotopy invariant.     Section \ref{S:GPP}, develops further properties of the fundamental groupoid,  showing how it behaves with respect to  product graphs and proving a modified Van Kampen Theorem \cite{Hatcher, VanKampen}.  We end with Section \ref{S:Comp}  describing a looped version of our groupoid, which generalizes the fundamental group of the exponential graph defined by Dochtermann \cite{Docht2}, and prove there is an  equivalence of categories between the looped groupoid of an exponential graph $H^G$ and the fundamental groupoid of the polyhedral hom space associated with $H^G$ studied in  \cite{Docht2, Kosolov1, Kosolov2, Kosolov3, KosolovShort, CompGH}.

\section{Background} \label{S:Back}

In this section, we summarize background material.  A more complete exposition, including examples, can be found in \cite{CS1}.      We work in the category $\Gph$ of  undirected graphs, without multiple edges.  Graph theory terminology and notation follows \cite{Bondy} and category theory terminology and notation follows \cite{riehlCTIC}.

\begin{definition} \cite{HN2004} The category of graphs $\Gph$ is defined by: \begin{itemize} \item An object is a  graph $G$,  consisting of a set of vertices $V(G) = \{ v_{\lambda}\}$ and a set $E(G)$ of edges connecting them.  Each edge is given by an unordered pair of vertices.     Any pair of vertices has at most one edge connecting them, and loops are allowed.     A connecting  edge will be notated by   $v_1 \con v_2$. 
\item 
An arrow in the category   $\Gph$ is a graph morphism $f:  G \to H$, given by a set map $f:  V(G) \to V(H)$ such that  if $v_1 \con v_2 \in E(G)$ then $f(v_1) \con f(v_2) \in E(H)$.  \end{itemize} \end{definition}  

In what follows, we will assume that  `graph' always refers to an object in $\Gph$.

In defining our fundamental groupoids, we will make use of the path graphs, in both looped and unlooped.  

\begin{definition}\label{D:path} \cite{Bondy, Docht1} Let $P_n$ be the path graph with $n+1$ vertices $\{ 0, 1, \dots, n\}$ such that $i \con i+1$.  Let $I_n^{\ell}$ be the looped path graph with  $n+1$ vertices   $\{ 0, 1, \dots, n\}$ such that $i \con i$ and  $i \con {i+1}$.  

$$ \begin{tikzpicture}
\node at (-.5,0){$P_n = $};
\draw[fill] (0,0) circle (2pt);
\draw (0,0) --node[below]{0} (0,0);

\draw[fill] (1,0) circle (2pt);
\draw (1,0) --node[below]{1} (1,0);

\draw[fill] (2,0) circle (2pt);
\draw (2,0) --node[below]{2} (2,0);

\node at (3,0){$\cdots$}  ;

\draw[fill] (4,0) circle (2pt);
\draw (4,0) --node[below]{$n$} (4,0);

\draw(0,0) -- (1,0);
\draw(1,0) -- (2,0);
\draw(2,0) -- (2.7,0);
\draw(3.3,0) -- (4,0);

\end{tikzpicture}
\hspace{1cm} \begin{tikzpicture}
\node at (-.5,0){$I_n^{\ell} = $};
\draw[fill] (0,0) circle (2pt);
\draw (0,0) --node[below]{0} (0,0);
\draw (0,0)  to[in=50,out=140,loop, distance=.7cm] (0,0);

\draw[fill] (1,0) circle (2pt);
\draw (1,0) --node[below]{1} (1,0);
\draw (1,0)  to[in=50,out=140,loop, distance=.7cm] (1,0);

\draw[fill] (2,0) circle (2pt);
\draw (2,0) --node[below]{2} (2,0);
\draw (2,0)  to[in=50,out=140,loop, distance=.7cm] (2,0);

\node at (3,0){$\cdots$}  ;

\draw[fill] (4,0) circle (2pt);
\draw (4,0) --node[below]{$n$} (4,0);
\draw (4,0)  to[in=50,out=140,loop, distance=.7cm] (4,0);

\draw(0,0) -- (1,0);
\draw(1,0) -- (2,0);
\draw(2,0) -- (2.7,0);
\draw(3.3,0) -- (4,0);

\end{tikzpicture}$$
\end{definition} 

\begin{definition}\label{D:walk} \cite{Bondy} A  {\bf walk} in $G$ of length $n$ is a morphism $\alpha:  P_n \to G$ from $\alpha(0) $  to $\alpha(n)$.  A {\bf looped walk} in 
$G$ of length $n$ is a morphism $\alpha:  I_n^{\ell} \to G$.   Note that we allow length $0$ walks, defined by a single vertex.   \end{definition}

We will usually  describe a walk by a list of image vertices $(v_0 v_1 v_2  \dots v_n)$  such that  $v_i \con v_{i+1}$.      In a looped walk,   all vertices along the walk are looped.  

\begin{definition}\label{D:concat} (\cite{CS1}, Definition 2.1)  Given a walk $\alpha:  P_{n} \to G$ from $x$ to $y$,  and a walk $\beta:   P_{m} \to G$ from $y$ to $z$, the   {\bf concatenation of walks}  $\alpha * \beta:  P_{m+n} \to G $ by    $$ (\alpha*\beta) (i) = \begin{cases} \alpha(i) &  \textup{ if } i\leq n \\ 
\beta({i-n}) &  \textup{  if } n < i \leq n +m\\ 
\end{cases} $$  Thus  the concatenation $(xv_1v_2 \dots v_{n-1}y)*(yw_1w_2 \dots w_{m-1}z) = (x v_1 v_2 \dots v_{n-1} y w_1 \dots w_{m-1} z).$
Contatenation of looped walks is defined in the same way.   \end{definition} 

Homotopies are defined using the product graph $G \times I_n^\ell$.

\begin{definition}\label{D:prod} \cite{HN2004} The (categorical) {\bf product graph} $G \times H$ is defined by: \begin{itemize}
\item  A vertex is a  pair $(v, w)$ where $v \in V(G)$ and $w \in V(H)$.
\item An edge is defined by $(v_1, w_1) \con (v_2, w_2) \in E(G \times H) $ for $v_1 \con v_2 \in E(G)$ and $w_1 \con w_2 \in E(H)$.   \end{itemize}
\end{definition}

\begin{definition}\cite{Docht1} \label{D:htpy} Given $f, g:  G \to H$, we say that $f$ is {\bf $\times$-homotopic} to $g$, written $f \simeq g$,  if there is a map $\Lambda: G \times I_n^{\ell}  \to H$ such that $\Lambda | _{G \times \{ 0\} } = f$ and $\Lambda | _{G \times \{ n\} } = g$.   We will say $\Lambda$ is a length $n$ homotopy.        \end{definition}

Other authors have considered alternate definitions of homotopies of graphs, and this is sometimes referred to as $\times$-homotopy to distinguish it.  
Since this is the primary version of homotopy that we will consider in this paper, we will also refer to it simply as 'homotopy'.  

Our fundamental groupoid will be defined by homotopy classes of walks using the following.

 \begin{definition}  \label{D:horelendpt} Suppose that $\alpha, \beta $ are walks in $G$  from $x$ to $y$.   We say $\alpha$ and $\beta$ are {\bf homotopic rel endpoints}   if all intermediate walks $\Lambda | _{G \times \{ i\} }$ are also walks from $x$ to $y$, so the endpoints of the walk remain fixed.  A similar definition holds for looped walks.   
\end{definition}

In \cite{CS1}, we analyzed the structure of homotopies concretely.  

\begin{definition} \label{D:spiderpair} Let $f, g:  G \to H$ be graph morphisms.  We say that $f$ and $g$ are a {\bf spider pair} if there is a single vertex $x$ of $G$ such that $f(y) = g(y)$ for all $y \neq x$.  If $x$ is unlooped there are no additional conditions, but if $x\con x \in E(G)$, then we require  that $f(x) \con g(x) \in E(H)$.    When we replace $f$ with $g$ we refer to it as a {\bf spider move}.  

\end{definition}

\begin{prop} \label{P:spider} (\cite{CS1}, Proposition 4.4: Spider Lemma) If $G$ is finite and  $f, g:  G \to H$ then  $f \simeq g$ if and only if is a finite sequence of spider moves connecting $f$ and $g$.  
\end{prop}

 We can also use this framework to analyze homotopy equivalences.   In the literature, homotopy equivalence has been linked to the idea of a {\bf fold}  \cite{HN2004, Docht2}.   This can be thought of as a special case of our spider moves.

\begin{definition}\label{D:fold}  If $G$ is a graph, we say that a morphism $f:   G\to G$ is a {\bf fold} if $f$ and the identity map are a spider pair.  
\end{definition}

\begin{proposition}  \label{P:fold}  If $f$ is a fold, then $f:  G \to Im(f)$  is a homotopy equivalence. 
\end{proposition}

 In the  literature, graphs that cannot  folded are referred to as {\bf stiff} graphs \cite{GMDG, BonatoCaR}.  For finite graphs, these can be taken as canonical representatives for graphs up to homotopy equivalence.

\begin{theorem}\label{T:skel}  (\cite{CS1}, Theorem 6.5) The stiff graphs are a skeletal subcategory of the homotopy category of finite graphs $\sf{hFGph}$ defined in \cite{CS1}, Definition 5.1. 
\end{theorem}

\section{The Walk Groupoid}\label{S:WalkG}
  Our walk groupoid will be based on walks in $G$ as defined in Definition \ref{D:walk}.   However, we want to be able to remove sections where we backtrack from these walks using the following concept.

\begin{definition}  Let $\alpha = (v_0 v_1 v_2 \dots  v_n)$ be a walk in $ G$.    We say that $\alpha$ is {\bf prunable} if $v_i = v_{i+2}$ for some $i$.   We  define a {\bf prune} of $\alpha$ to be given by a walk  $\alpha'$ obtained by deleting the vertices $v_i$ and $v_{i+1}$ from the walk when $v_i = v_{i+2}$:     
if 
$$ \alpha = (v_0 v_1 v_2 \dots v_{i-1} v_{i} v_{i+1} v_i v_{i+3}  \dots v_n) $$  
then the prune  of $\alpha$ is  $$ \alpha' = (v_0 v_1 v_2 \dots v_{i-1}  v_i v_{i+3}  \dots v_n) $$  

\end{definition}

We define an equivalence relation on walks in $G$ generated by the prunes.  Concretely,  $\alpha \simeq \beta$ if there is a finite sequence of prunings between them: $\alpha= \gamma_0 \simeq \gamma_1 \simeq \gamma_2 \simeq \dots \simeq \gamma_{k-1} \simeq  \gamma_k = \beta$ where either $\gamma_i$ is a prune of $\gamma_{i+1}$ or $\gamma_{i+1}$ is a prune of $\gamma_i$.

\begin{obs}
Since any prune always removes two edges, the parity of a prune equivalence class is well-defined and each prune class of walks consists of all even length or all odd length walks. 
\end{obs}

Each prune equivalence class has a unique non-prunable representative, as shown by the next two results.

\begin{proposition} \label{P:prune}
Repeated pruning of a walk results in a unique non-prunable walk.  
\end{proposition}
\begin{proof}
We proceed via induction.  If $\alpha$ is length $0$ or $1$, then there are no prunings possible and hence $\alpha$ is itself the unique non-prunable walk. 
Now consider a walk $\alpha:  P_n\to G$.  If there exists a unique $i$ such that $v_i = v_{i+2}$, then pruning $\alpha$ results in a unique $\alpha'$ of length $n-2$ and we are done by induction.  

Now suppose that there are two values  $i, j$ such that $v_{i } = v_{i+2}$ and $v_j = v_{j+2}$, and hence two possible prunings of $\alpha$.  We will show that either order of pruning will lead to the same result.  
 Without loss of generality, assume $i<j$.  If $i+1 < j$, then  the path $\alpha$ is of the form $$ (v_0 v_1 \dots, v_i v_{i+1} v_{i+2} \dots v_j v_{j+1} v_{j+2} \dots v_n).$$ (where we may have $v_{i+2} = v_j$ if $i+2 = j$).   It is clear that the results of pruning at $i$ and then $j$ are the same as pruning at $j$ and then $i$.

If $j+ i=1$, then $\alpha $ is of the form 
   $$\alpha =  (v_0 v_1 \dots, v_{i-1} v_i v_{i+1} v_{i} v_{i+1} v_{i+4}   \dots v_n)$$ 
    Pruning at $i$ removes the first $v_iv_{i+1}$ pair, while pruning at $j = i+1$ removes the $v_{i+1}v_i$ pair.  Both pruning orders  result in  $$ \alpha' =  (v_0 v_1 \dots, v_{i-1}  v_{i} v_{i+1} v_{i+4}   \dots v_n)$$  
    Thus by induction any choice of successive prunings on $\alpha$ will eventually result in the same non-prunable walk.  

\end{proof}

  \begin{corollary} \label{C:rep}
  Each prune class of walks has a unique non-prunable representative.  
  \end{corollary}
  
  \begin{proof}
  If we have two non-prunable walks $\alpha, \beta$ such that $[\alpha] = [\beta]$ then then there is a sequence of forward and backward prune moves connecting them:  $ \alpha \longleftarrow \gamma_1 \longrightarrow \gamma_2 \longleftarrow \gamma_3 \longrightarrow \dots \gamma_k \longrightarrow \beta$ where each arrow represents a sequence of prunes in the indicated direction.  We induct on $k$:  if $k=1$ then we have   $\alpha \longleftarrow \gamma_1 \longrightarrow  \beta$, and Proposition \ref{P:prune} ensures that $\alpha = \beta$ since they both result from prunings of the same path $\gamma_1$.    If $k>1$, then consider the left portion of the sequence of prune moves  $\alpha \longleftarrow \gamma_1 \longrightarrow \gamma_2 $:  letting  $\gamma'$ be the walk that results from completely pruning $\gamma_2$, we have that $\alpha = \gamma'$ by Proposition \ref{P:prune} again.  But then we have a  sequence of prune moves of length $k-2$ connecting $\gamma'$ to $\beta$, and by our inductive hypothesis we can say that $\gamma' = \beta$.  
  \end{proof}
Our walk groupoid will consist of prune classes of walks under concatenation, as defined in Definition  \ref{D:concat}.     To define this, we need to show:  

\begin{lemma} \label{P:concat}  Concatenation is well-defined on prune classes.    \end{lemma}

\begin{proof}The endpoints of any representatives of a prune class are always the same, and so the start and end vertices are well defined on prune classes.  If $\alpha$ prunes to $\alpha'$ and $\beta$ prunes to $\beta'$ then $\alpha * \beta$ prunes to $\alpha' * \beta'$, and  Proposition \ref{P:prune} that the order in which the prunes are done will not matter.

\end{proof}

\begin{definition} Let $\PX$ be the walk groupoid defined by the following:  

\begin{itemize}
    \item objects of $\PX$ are vertices of the graph $G$
    \item an arrow from $v_0$ to $v_n$ in $\PX$ is given by a prune class of walks from $v_0$ to $v_n$ 
    \item composition of arrows is defined using concatenation of walks as in Definition  \ref{D:concat}.
\end{itemize}
\end{definition}
 To see that $\PX$ is a groupoid, observe that concatenation of walks is associative (\cite{CS1}, Lemma 2.17)  and the length $0$ walk at a vertex $v$ gives an identity arrow from $v$ to $v$   (\cite{CS1}, Observation 2.16.)  Lastly, given any walk $\alpha =  (v_0 v_1 v_2 \dots v_{n-1} v_n)$  we define $\alpha^{-1} = (v_n v_{n-1} \dots v_2 v_1 v_0)$;  it is easy to see that $\alpha * \alpha^{-1}$ prunes down to a length $0$ walk. 
 
 Since every prune class has a unique non-prunable representative, we can also think of this groupoid as having arrows given by non-prunable walks, where the composition operation is given by concatenation followed by pruning.  

To get a group from this groupoid, we can fix a vertex $v$  and consider the isotropy group of the object $v$ consisting of all arrows from $v$ to $v$.     Any choice of vertices in the same connected component of $G$ will result in isomorphic  groups.

\begin{example}
Consider the graph $C_5$

$$
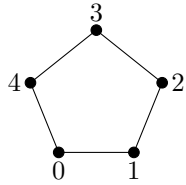
\begin{tikzpicture}
\draw[fill] (0,0) circle (2pt);
\draw (0,0) --node[below]{$0$} (0,0);
\draw[fill] (1,0) circle (2pt);
\draw (1,0) --node[below]{$1$} (1,0);
\draw[fill] (1.3755,0.9268) circle (2pt);
\draw (1.3755,0.9268) --node[right]{$2$} (1.3755,0.9268);
\draw[fill] (0.5,0.9268+.7) circle (2pt);
\draw (0.5,0.9268+.7) --node[above]{$3$} (0.5,0.9268+.7);
\draw[fill] (-0.3755,0.9268) circle (2pt);
\draw (-0.3755,0.9268) --node[left]{$4$} (-0.3755,0.9268);

\draw (0,0) -- (1,0)--(1.3755,0.9268)--(0.5,0.9268+.7)--(-0.3755,0.9268)--(0,0);

\end{tikzpicture}$$

The groupoid $\W C_5$ has  objects given by the vertex set $\{ 0, 1, 2, 3, 4\}$.  We   consider the isotropy group at $0$, given by  prune  classes of  walks from $0$ to $0$.  If any walk reverses orientation  and goes from clockwise to counterclockwise or vice versa, there will be a subwalk  which can be pruned, and so  all non trivial walks from $0$ to $0$ may be represented by strictly clockwise or counterclockwise walks,  generated by $(043210)$ and $(012340)$ respectively.  Since the concatenation of these walks  prune to the identity walk $(0)$, these are free generators for the group, which is isomorphic to $\mathbb{Z}$.  Since $C_5$ is connected, the groupoid $\W C_5$ retracts down onto this group.

\end{example}

\begin{obs}
The walk groupoid $\W G$ is a full subcategory of the fundamental groupoid of the graph $G$ considered as a 1-dimensional topological space. If $G$ is connected then topologically,  $G$ is homotopy equivalent to a wedge of circles.
Fixing a  basepoint vertex $v$  and considering the walk group defined by all arrows from $v$ to $v$, we can get a presentation of this group as the free group generated by edges which do NOT belong to a spanning tree for $G$:  each generating loop from $v$ to $v$ consists of concatenating the given edge $w \sim w'$ with the unique paths from $v$ to $w$ and then from $w'$ to $v$.  The full groupoid $\W G$ will have this walk group as a retract, with isomorphic isotropy groups at all vertices.   Arrows between any two objects $w, w'$ may be recovered by choosing a walk from $w$ to $w'$, and then concatenating with elements of the group with basepoint $w$.  

If $G$ is disconnected, the walk groupoid will be the coproduct of groupoids with the structure described above for each component of the graph. 
\end{obs}

\begin{theorem} \label{T:walkfunctor} $\PX$ defines a functor from $\Gph$ to $\sf{Groupoids}$, the category of groupoids.  

\end{theorem}

\begin{proof}

If we have a graph homomorphism $\phi:  G \to H$, we can define a functor $\phi_*:  \W G \to \W H$ by $\phi_*(v) = \phi(v)$ on objects, and  $\phi_*(\alpha) = \phi_*(v_0 v_1 v_2 \dots v_n)$ is the walk in $H$ defined by $(\phi(v_0) \phi(v_1) \phi(v_2) \dots \phi(v_n))$;  the fact that $\phi$ is a graph homomorphism ensures that this a walk in $H$.   If $\alpha$ prunes to $\alpha'$, then  $\phi^*(\alpha)$ also prunes to $\phi^*(\alpha')$, and so   this can be extended to a moprhism of groupoids $\PX \to \mathcal{W}H$ by using $\phi$ on objects and $\phi^*$ on arrows, and it is easy to see that this respects concatenation, and hence the composition in the groupoid.  

 To verify functoriality, observe that   if $id:  G \to G$ is the identity, then $id_*$ is the identity map on groupoids; and if  $\phi: G \to H$ and $\psi:  H \to K$, then $(\psi \phi)_*$ is the same as $\psi_* \phi_*$ since they are both defined by $ (\psi \phi(v_0) \psi \phi(v_1) \psi\phi(v_2) \dots \psi \phi(v_n))$.  

\end{proof}

\section{The Fundamental Groupoid}  \label{S:fgp}

Our walk groupoid does not consider homotopy.  In this section, we define a fundamental groupoid of homotopy classes of walks, where homotopy is taken rel endpoints as in Definition \ref{D:horelendpt}.

\begin{definition} \label{D:fg} Let $\Pi(G)$ be the fundamental groupoid of $G$  defined by the following:  

\begin{itemize}
    \item objects of $\Pi(G)$ are vertices of the graph $G$
    \item an arrow from $v_0$ to $v_n$ in $\Pi(G)$ is given by a prune class of walks from $v_0$ to $v_n$, up to  homotopy rel endpoints
    \item composition of arrows is defined using concatenation of walks  as in \cite{CS1}, Definition 2.14.  
\end{itemize}
\end{definition}

\begin{example}\label{E:prune}

 Let $\alpha={(acbce)}$ and  $\beta={(ade)}$ be walks in the graph below. Then  $\alpha=  \beta\in \Pi(G)$  since we have a prune of   $\alpha$ to $\alpha' = {(ace)}$ and then a spider move to  $\beta={(ade)}$.

$$
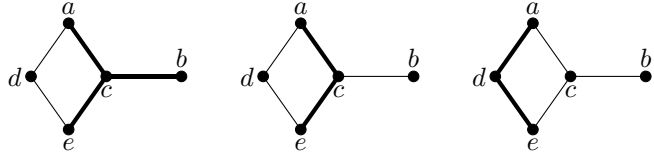
\begin{tikzpicture}
\draw[fill] (0,0) circle (2pt);
\draw (0,0) --node[left]{$d$} (0,0);
\draw[fill] (.5,.707) circle (2pt);
\draw (.5,.707) --node[above]{$a$} (.5,.707);
\draw[fill] (.5,-.707) circle (2pt);
\draw (.5,-.707) --node[below]{$e$} (.5,-.707);
\draw[fill] (1,0) circle (2pt);
\draw (1,0) --node[below]{$c$} (1,0);
\draw[fill] (2, 0) circle (2pt);
\draw (2,0) --node[above]{$b$} (2,0);

\draw (0,0) -- (.5,.707) -- (1,0) -- (.5,-.707)--(0,0);
\draw (1,0)--(2,0);

\draw[ultra thick] (.5, .707) -- (1, 0) -- (2, 0) -- (1,0)--(.5, -.707);

\end{tikzpicture}
\ \ \ \ 
\begin{tikzpicture}
\draw[fill] (0,0) circle (2pt);
\draw (0,0) --node[left]{$d$} (0,0);
\draw[fill] (.5,.707) circle (2pt);
\draw (.5,.707) --node[above]{$a$} (.5,.707);
\draw[fill] (.5,-.707) circle (2pt);
\draw (.5,-.707) --node[below]{$e$} (.5,-.707);
\draw[fill] (1,0) circle (2pt);
\draw (1,0) --node[below]{$c$} (1,0);
\draw[fill] (2, 0) circle (2pt);
\draw (2,0) --node[above]{$b$} (2,0);

\draw (0,0) -- (.5,.707) -- (1,0) -- (.5,-.707)--(0,0);
\draw (1,0)--(2,0);

\draw[ultra thick] (.5, .707) -- (1, 0) --(.5, -.707);

\end{tikzpicture}
\ \ \ \ 
\begin{tikzpicture}
\draw[fill] (0,0) circle (2pt);
\draw (0,0) --node[left]{$d$} (0,0);
\draw[fill] (.5,.707) circle (2pt);
\draw (.5,.707) --node[above]{$a$} (.5,.707);
\draw[fill] (.5,-.707) circle (2pt);
\draw (.5,-.707) --node[below]{$e$} (.5,-.707);
\draw[fill] (1,0) circle (2pt);
\draw (1,0) --node[below]{$c$} (1,0);
\draw[fill] (2, 0) circle (2pt);
\draw (2,0) --node[above]{$b$} (2,0);

\draw (0,0) -- (.5,.707) -- (1,0) -- (.5,-.707)--(0,0);
\draw (1,0)--(2,0);

\draw[ultra thick] (.5, .707) -- (0,0) -- (.5, -.707);

\end{tikzpicture}
$$

\end{example}

In order to verify that this definition gives us a well-defined groupoid, we check the following.

\begin{proposition} \label{P:concatwd} Concatenation is well-defined on elements of $\Pi(G)$.  \end{proposition}

\begin{proof}    We have already shown that it is well-defined with respect to pruning in Proposition \ref{P:concat}, so we need to check that it respects homotopy.  
Suppose that we have walks that are homotopic rel endpoints:   $\alpha \simeq \alpha'$ and $\beta \simeq \beta'$.   Then there is a sequence of spider moves connecting $\alpha$ to $\alpha'$, and $\beta$ to $\beta'$.  So we can produce a sequence of spider moves connecting $\alpha * \beta$ to $\alpha' * \beta'$ by holding $\beta$ fixed and moving $\alpha * \beta$ to $\alpha' * \beta$, and then holding $\alpha'$ fixed and moving $\alpha'*\beta$ to $\alpha' * \beta'$.

\end{proof}

\begin{theorem} \label{T:gpoid}  $\Pi(G)$ defines a groupoid.  \end{theorem}

 \begin{proof}
 The concatenation operation is associative as shown in \cite{CS1}, Lemma 2.17, and given a vertex $v \in G$  we have the length $0$ walk $(v)$ acting as an identity element by \cite{CS1}, Observation 2.16.   
 If $\alpha = (v_0 v_1 \dots v_{n-1} v_n)$ then we can define an inverse $\alpha^{-1} = (v_n v_{n-1} \dots v_1 v_0)$ .  Then $$\alpha * \alpha^{-1}  =  (v_0 v_1 v_2 \dots v_{n-2} v_{n-1} v_n v_{n-1} v_{n-2} \dots v_2 v_1 v_0  ).$$  Successive pruning operations will reduce this to the identity walk  $(v_0)$.

 \end{proof}

 We can describe the arrows of $\Pi(G)$ more concretely with the following result.  
\begin{lemma} \label{L:moveprune}
Let $\alpha$ be  prunable at $i$, so $v_i = v_{i+2}$ and $\alpha$ has the form    $$ (v_0 v_1 \dots, v_{i-1} v_i v_{i+1} v_{i} v_{i+3}   \dots v_n)$$  Then $\alpha$ is homotopic rel endpoints to the walk   $$ (v_0 v_1 \dots, v_{i-1}  v_{i} v_{i+2}   \dots v_n v_{n-1} v_n).$$  
\end{lemma}
\begin{proof}  As maps from $P_n \to G$, we apply successive spider moves to $\alpha$ to move the repeated vertex down the walk:  
\begin{align*} \alpha & =   (v_0 v_1 \dots, v_{i-1} v_i v_{i+1} v_{i} v_{i+3} v_{i+4}  v_{i+5}  \dots v_n) \\ & \simeq   (v_0 v_1 \dots, v_{i-1} v_i v_{i+3} v_{i} v_{i+3}  v_{i+4} v_{5} \dots v_n) \\ & \simeq  (v_0 v_1 \dots, v_{i-1} v_i v_{i+3} v_{i+4} v_{i+3} v_{i+4}  v_5 \dots v_n) \end{align*}  Repeatedly applying spider moves  will shift the repeat down to the end of the walk.  

\end{proof}

Thus for arrows of $\Pi(G)$, we can consider only prunes of the last two edges.  This allows us to identify arrows with  homotopy classes of walks of infinite length, which eventually stabilize and end with a string of alternating vertices $v_n v_{n-1} v_n v_{n-1} v_n v_{n-1} v_n \dots $.    Two such walks $\alpha, \beta$ will be equivalent if there is some extension of each which become homotopic rel endpoints: if  $\alpha =  (v_0 v_1 \dots  v_n)$ and $\beta =(w_0 w_1   \dots w_m) $ then there exists extensions  $ (v_0 v_1 \dots  v_n, v_{n-1}  v_n v_{n-1} \dots  v_n v_{n-1}v_n)$ and  $(w_0 w_1   \dots w_m w_{m-1}   \dots  w_m w_{m-1}w_m)$ which are homotopic rel endpoints.

As with our walk gropuoid, our fundamental groupoid defines a functor from $\Gph$ to groupoids.

\begin{theorem} \label{T:functor}   $\Pi$ defines a functor from $\Gph$ to groupoids. \end{theorem}

\begin{proof}  
Suppose that $\phi:  G \to H$ is a graph homomorphism, and define $\phi_*:  \Pi(G) \to \Pi(H)$ as in Theorem \ref{T:walkfunctor}.   We have shown this is a functor from $\W G$, and so respects prune classes.    If  $\alpha$ and $\beta$ are homotopic rel endpoints, there is a sequence of spider moves connecting them, shifting  one vertex   $(v_0 v_1 \dots v_{i-1} v_i v_{i+1} \dots v_n)$ to $(v_0 v_1 \dots v_{i-1} \hat{v}_i v_{i+1} \dots v_n)$, and applying $\phi_*$ will give a sequence of walks where each pair similarly differs by a single vertex, and hence is a sequence of spider moves.  So $\phi_*(\alpha)$ will be homotopic rel endpoints to $\phi_*(\beta)$.  

Functoriality also follows from the argument from Theorem \ref{T:walkfunctor}.  

\end{proof}

We now show that the fundamental groupoid defines a homotopy invariant for finite graphs with no isolated vertices.  

\begin{theorem}\label{T:natiso}  If $G, H$ are finite graphs with no isolated vertices, and  $\phi, \psi:  G \to H$ are homotopic, then there is a natural isomorphism from $\phi_*$ to $\psi_*$.   \end{theorem}

\begin{proof} 
By the Spider Lemma \ref{P:spider} it is enough to consider the case when $\phi$ and $\psi$ are connected by a spider move.    So assume that $\phi$ and $\psi$ agree on every vertex except one, say $v$.    To define the  natural isomorphism, we must choose an arrow $\gamma_w$ in  $\Pi(H)$ from $\phi_*(w)$ to $\psi_*(w)$ for each vertex $w$ in $G$.  For $w \neq v$, we have $\phi(w)=\psi(w)$ and so we choose the length $0$ identity walk  on $\phi(w) = \psi(w)$.  For the vertex where they differ, recall that $G$ has no isolated vertices, and choose  $w \in N(v)$ and define the walk $\gamma_v$ by  
 $(\phi(v)\phi(w) \psi(v))$.   If $v = w$,  then $v $ is looped, so $\phi(v)$ and $\psi(v)$ are connected and $\gamma_v$ is defined by  $\phi(v) \phi(v) \psi(v)$.  The  walk in $\Pi(H)$ is independent of choice of $w$,  since all choices are homotopic rel endpoints, with a spider move connecting them.    
 
 To check the required naturality square, we consider a walk $\alpha:  w \to w'$ in $G$.  If neither $w, w'$ are equal to $v$, then the naturality square is easily checked by observing that $\phi(\alpha) = \psi(\alpha)$ up to homotopy rel endpoints (they will be exactly equal if $v$ is not included in $\alpha$, and homotopic otherwise).  If we have a path that starts at $v$, so $\alpha = (v u_1 u_2 u_3 \dots w')$ then we need to compare $\gamma_v * \psi(\alpha)$ with $\phi(\alpha)$.    Now if we choose $u_1$ as the neighbour of $v$ for creating $\gamma_v$, we get  $\gamma_v * \psi(\alpha) =  \phi(v) \psi(u_1) \psi(v) \psi(u_1) \psi(u_2) \dots \psi(w')$ which prunes to $\phi(v) \psi(u_1) \psi(u_2) \dots \psi(w')$.  But then $\psi(u_i) = \phi(u_i)$ (unless $u_i = v$, in which case they are connected by a spider move as in the first case) so this is equal to $\phi(\alpha)$.  A similar check shows naturality for paths that end at $v$.  

\end{proof}

\begin{corollary} $\Pi$ defines a  2-functor from the 2-category  $\Gph$ of graphs with no isolated vertices to the 2-category of groupoids, functors and natural transformations.   Thus $\Pi$ passes to a functor from the homotopy category $\sf {hFGp}h$ of finite graphs with no isolated vertices to the category of groupoids and functors up to natural isomorphism.  
\end{corollary}

\begin{corollary} The equivalence class of the category  $\Pi(G)$ is a homotopy invariant for finite graphs with no isolated vertices.   \end{corollary}

\begin{proof} If $\phi: G \to H$ is a homotopy equivalence, then there is $\psi:  H \to G$ such that $\phi \psi \simeq id$ and $\psi \phi \simeq id$.  Then there is a natural isomorphism from $\Pi(G)$ to $\psi_*\phi_*\Pi(G)$, and from $\Pi(H) $ to $\phi_*\psi_*\Pi(H)$ and so $\phi_*$ and  $\psi_*$ are equivalences of categories between $\Pi(G)$ and $\Pi(H)$.  

\end{proof} 

\begin{corollary} If $G$ has no isolated vertices, and $G'$ denotes the (unique) stiff graph which is homotopy equivalent to $G$,  then the fundamental groupoid $\Pi(G)$ is equivalent to the fundamental groupoid of   $\Pi(G')$.  
\end{corollary}

\begin{obs}\label{O:FG}  We have chosen to work with the fundamental groupoid here.  It is easy to recover a more familiar fundamental group by choosing a 
basepoint vertex $v$ in $G$, and looking at the group $\Pi_1(G, v)$ of all arrows in $\Pi(G)$ which start and end at $v$;  this is just the isotropy subgroup of $v$ in the groupoid.    Because $\Pi(G)$ is a groupoid, we know that we have an isomorphism between the isotropy groups of any two choices of vertex in the same component of $G$, and the groupoid of any component retracts onto the isotropy group of the chosen basepoint.   
\end{obs}

\begin{example}  
Let $G$ be the graph  from Example \ref{E:prune}:

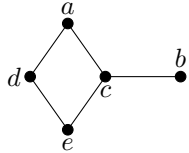
$$\begin{tikzpicture}
\draw[fill] (0,0) circle (2pt);
\draw (0,0) --node[left]{$d$} (0,0);
\draw[fill] (.5,.707) circle (2pt);
\draw (.5,.707) --node[above]{$a$} (.5,.707);
\draw[fill] (.5,-.707) circle (2pt);
\draw (.5,-.707) --node[below]{$e$} (.5,-.707);
\draw[fill] (1,0) circle (2pt);
\draw (1,0) --node[below]{$c$} (1,0);
\draw[fill] (2, 0) circle (2pt);
\draw (2,0) --node[above]{$b$} (2,0);

\draw (0,0) -- (.5,.707) -- (1,0) -- (.5,-.707)--(0,0);
\draw (1,0)--(2,0);

\end{tikzpicture}$$

By Theorem \ref{T:natiso} we have that $\Pi(G)\cong\Pi(K_2)$, since  $K_2$ is the stiff homotopy equivalent representative of $G$.  

$$\begin{tikzpicture}
\draw[fill] (0,0) circle (2pt);
\draw (0,0) --node[below]{$0$} (0,0);
\draw[fill] (1,0) circle (2pt);
\draw (1,0) --node[below]{$1$} (1,0);

\draw (0,0) -- (1,0);

\end{tikzpicture}$$

The objects of  $\Pi(G)$ are  the vertices $0, 1$ and the arrows are  identity arrows given by length 0 walks at $0$ and $1$, and the length 1 walks between them.  Any other walk would consist of alternating 0's and 1's, and may thus be pruned to a length 1 walk.  Choosing a basepoint, we get a trivial fundamental group.  
\end{example}

As with the walk group, we can develop a  concrete desciption of the fundamental group $\Pi_1 (G, v)$ of all arrows in $\Pi(G)$ from a chosen basepoint $v$ to $v$ in terms of generators and relations.  

\begin{theorem} \label{T:genrel} If $G$ is connected, 
$\Pi_1 (G, v) = \W_v^v G/D$ where  $D$ is the normal subgroup generated by all diamonds given by walks of the form $\gamma * (v_1 v_2 v_3 v_4 v_1)* \gamma^{-1} $ for $\gamma $ a walk  from $v$ to $v_1$.   
       
\end{theorem}

\begin{proof}  Any walk in $D$ is nulhomotopic, since if $d= (v_1 v_2 v_3 v_4 v_1)$. then there is a spider move  to $d' = (v_1 v_2 v_2 v_2 v_1)$ which prunes to the empty walk.  

Conversely, if two walks from $v$ to $v$ are homotopic, then there is a sequence of spider moves connecting them.  Each spider move will shift one vertex, so consider \begin{align*} \alpha & = (v w_1 w_2 \dots w_{i-1} w_i w_{i+1} \dots w_{n-1} v) \\ \beta & = (v w_1 w_2 \dots w_{i-1} \hat{w}_i w_{i+1} \dots w_{n-1} v) \end{align*}.  Define  $\gamma    = (v w_1 w_2 \dots w_{i-1}) $ and the diamond $d = (w_{i-1} \hat{w}_i w_{i+1} {w_i} w_{i-1})$  Then $\gamma d \gamma^{-1} * \alpha  $ prunes to $\beta$ and so $\alpha$ and $\beta$ are equivalent in $\W G / D$.  
\end{proof}

\begin{corollary}

If $G$ is connected then   the fundamental group $\Pi_1 (G, v)$   is defined by $F/D$ where $F$ is the free group generated by all edges of $G$ which are NOT contained in a spanning tree $T$, and $D$ is the normal subgroup generated by all diamonds.  

\end{corollary}
\begin{example}
Let $G$ be the graph depicted below: 

$$   
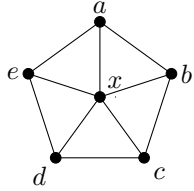
\begin{tikzpicture}
\draw (0,1)\foreach \x in {1,...,5}{--({sin(72*\x)}, {cos(72*\x)})};

\foreach \x in {1,...,5}{\draw (0,0)--({sin(72*\x)}, {cos(72*\x)});}

\draw[fill] ({sin(72*0)}, {cos(72*0)}) circle (2pt);
\draw[fill] ({sin(72*1)}, {cos(72*1)}) circle (2pt);
\draw[fill] ({sin(72*2)}, {cos(72*2)}) circle (2pt);
\draw[fill] ({sin(72*3)}, {cos(72*3)}) circle (2pt);
\draw[fill] ({sin(72*4)}, {cos(72*4)}) circle (2pt);
\draw[fill,] (0, 0) circle (2pt);

\draw ({sin(72*0)}, {cos(72*0)}) -- node[above]{$a$} ({sin(72*0)}, {cos(72*0)});
\draw ({sin(72*1)}, {cos(72*1)}) -- node[right]{$b$} ({sin(72*1)}, {cos(72*1)});
\draw ({sin(72*2)}, {cos(72*2)}) -- node[below right]{$c$} ({sin(72*2)}, {cos(72*2)});
\draw ({sin(72*3)}, {cos(72*3)}) -- node[below left]{$d$} ({sin(72*3)}, {cos(72*3)});
\draw ({sin(72*4)}, {cos(72*4)}) -- node[left]{$e$} ({sin(72*4)}, {cos(72*4)});
\draw (.2, 0) -- node[above]{$x$} (.2, 0);

 \end{tikzpicture}  $$
 
 A spanning tree is defined by all the edges connected to the central vertex $x$, and so one presentation for the fundamental groupoid is given by the free group on $5$ generators $e_1= (xabx), e_2 = (xbcx), e_3 = (xcdx), e_4 = (xdex), e_5 = (xeax)$ modulo the normal subgroup generated by the diamonds   $(xabcx), (xbcdx), (xcdex), (xdeax), (xeabx)$.  But these are equal to $e_1e_2, e_2 e_3, e_3e_4, e_4e_5, e_5e_1$.  This means that $e_2 = e_1^{-1}$, and $e_2 = e_3^{-1}$, etc.  Thus  we find that $e_1 = e_3 = e_5$ and $e_2 = e_4 = e^{-1}_1$ and the group is generated by a single generator $e_1$ under the relationship $e_1^2 =1$.  Thus we have the fundamental group defined by $\mathbb{Z}/2$.   This shows that unlike the walk groupoid, our fundamental group can contain torsion.  
 \end{example}
\section{Fundamental Groupoid of Products and Pushouts} \label{S:GPP}

In this section we further examine the structure of our fundamental groupoid $\Pi (G)$, looking deeper into the parity structure of even and odd loops and analyzing the fundamental groupoid of product graphs.  
 
 Let $T$ be the terminal object of $\Gph$ which has one vertex and one loop edge $\tau$ \cite{Demitri, HN2004}.   Then $\Pi(G)$ is a groupoid with one object, hence a group, and it has two arrows:  the identity arrow given by the length $0$ walk, and the length $1$ walk $(\tau)$.  The walk $(\tau\tau)$ can be pruned to the identity empty walk, so as a group we have $\tau^2=id$ and so $\Pi(T)$ is isomorphic to $\mathbb{Z}/2$.
 
The parity structure of the fundamental groupoid can be linked to the fact that the terminal object of   $\Gph$  has a groupoid which is not the terminal identity groupoid.   Every graph $G$  has a unique canonical morphism to $T$, and so we have a groupoid morphism $\Pi(G) \to \Pi(T)$ and our fundamental groupoids live in the category of groupoids over $\mathbb{Z}/2$, with all morphisms of groupoids induced by graph maps respecting this structure.  Explicitly, we have that the canonical map $\Pi(G) \to {\mathbb{Z}/2}$ takes even walks to $id$ and odd walks to $\tau$, and every map from $\Pi(G) \to \Pi(H)$ that comes from a graph map $G \to H$ will commute with the map to $\Pi(T)$ and hence preserve parity.   We can define the even subgroupoid $ev(\Pi(G)) = p^{-1}(id)$.  
 
The product $G \times H$ is the pullback over the terminal object 
 $$ \xymatrix{ G \times H \ar[r] \ar[d] & H \ar[d]  \\ 
 G \ar[r] & T }
$$ 
Functoriality says that the projections $p_1:  G \times H \to G$ and $p_2:  G\times H \to H$ give maps $\Pi(G\times H) \to \Pi(G)$ and $\Pi(G\times H) \to \Pi(H)$ and so we will  have the following diagram:
 $$ \xymatrix{ \Pi(G \times H)  \ar@{-->}[rd] \ar[rdd] \ar[rrd] & \\ & \Pi(G) \times_{\mathbb{Z}/2} \Pi(H) \ar[d] \ar[r] & \Pi(H) \ar[d]  \\ 
& \Pi(G) \ar[r] & \Pi(T) = \mathbb{Z}/2 }$$
where $\Pi(G) \times_{\mathbb{Z}/2} \Pi(H)$ denotes the pullback groupoid.  
Explicitly, the  pullback  is defined as follows:  the objects are the product of the objects of $\Pi(G)$ and $\Pi(H)$, and arrows are given by $(\alpha, \beta) | p_1(\alpha) = p_2(\beta)$,  which means $(\alpha, \beta)$ such that the parity of the walks are the same.

\begin{theorem}
The induced map $\Phi:  \Pi(G \times H) \to  \Pi(G) \times_{\mathbb{Z}/2} \Pi(H)$ is an isomorphism of groupoids.  
\end{theorem}

\begin{proof}

Objects of $\Pi(G\times H)$ are given by vertices of $G \times H$ which is the set $V(G) \times V(H)$, the objects of $\Pi(G) \times_{\mathbb{Z}/2} \Pi(H)$, so this  is an isomorphism on objects.  

  On arrows,  the map is defined by  $\Phi(\omega) = (\alpha, \beta)$ where  $p_1(\omega) = \alpha$ in $G$ and $p_2(\omega) = \beta$ in $H$. We need to show that this is both  full and faithful (injective and surjective).    To show that $\Phi$  is surjective on arrows, suppose we have    $(\alpha, \beta) \in \Pi(G) \times_{\mathbb{Z}/2} \Pi(H)$  given by $\alpha \in \Pi(G)$ and $\beta \in \Pi(H)$ with the same parity.  Take the shorter one and repeat the last two vertices to extend so that both representative walks have the same length.  This will create  a walk $\omega$ in $G \times H$ such that $\Phi(\omega) = (\alpha, \beta)$.   
  
  If $\Phi(\omega) = \Phi(\omega')$ then $\alpha = \alpha' $ in $\Pi(G)$ and $\beta = \beta' $ in $\Pi(H)$.  This means there are extensions of $\alpha, \alpha'$ which are homotopic rel endpoints in $G$;  and extensions of $\beta, \beta'$ which are homotopic rel enpoints in $H$.  We are assuming that these have the same parity, and so by extending further,  we may assume all are the same length.  Then we can combine the homotopies $H \times H'$ to get a homotopy in $G \times H$.  This shows that the functor is also injective on arrows.

\end{proof}

\begin{example}\label{Example:ParityGroupoid}
Let $G=P_2, H=K_2$ and consider $G\times H$:

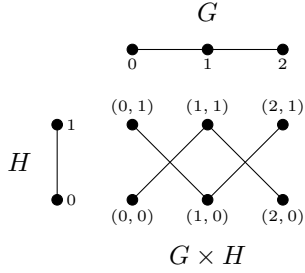
$$\begin{tikzpicture}

\draw (-1,0)--(-1,1);
\draw (0,2)--(2,2);
\foreach \x in {0,...,1}{\draw[fill] (-1, \x) circle (2pt); }
\foreach \x in {0,...,1}{\draw (-1, \x) --node[right]{\tiny $\x$} (-1,\x); }
\foreach \x in {0,...,2}{\draw[fill] (\x, 2) circle (2pt); }
\foreach \x in {0,...,2}{\draw (\x, 2) --node[below]{\tiny $\x$} (\x, 2); }

\foreach \x in {0,...,2}{\foreach \y in {0,...,1}{\draw[fill] (\x, \y) circle (2pt); } }
\foreach \x in {0,...,2}{\draw (\x, 1) --node[above]{\tiny $(\x,1)$} (\x,1); }
\foreach \x in {0,...,2}{\draw (\x, 0) --node[below]{\tiny $(\x,0)$} (\x,0); }

\draw (0,0) -- (1,1) -- (2,0);
\draw (0,1) -- (1,0) -- (2,1);

\node at (-1.5, .5){$H$};
\node at (1, 2.5){$G$};
\node at (1, -0.75){$G\times H$};

\end{tikzpicture}$$

There is an odd length walk from $(0,0)$ to $(1,1)$ since there is an odd length walk from $0$ to $1$ in both $G$ and $H$.  Similarly, there is an even length walk from $(0,0)$ to $(2,0)$.
However, there is no walk from $(0,0)$ to $(1,0)$, since the  walks from $0$ to $1$ in $G$ and $0$ to $0$ in $H$ have different parity.

\end{example}

If we consider reflexive graphs (where all vertices have loops) then the parity plays less of a role and our fundamental groupoid winds up with odd and even portions isomorphic to each other.  To make this precise, we look at a product groupoid  $X \times \mathbb{Z}/2$.    The objects of this product are the same as the objects of $X$ and the arrows are given by $(\alpha, id)$ and $(\alpha, \tau)$.  

\begin{proposition} Suppose that $G$ is a reflexive graph, and $\Pi(G) = \Pi$ is its fundamental groupoid, and $E = ev(\Pi(G))$ its even subgroupoid.  Then $\Pi \simeq E \times \mathbb{Z}/2$.

\end{proposition}

\begin{proof}
Define $\Psi:  \Pi \to E \times \mathbb{Z}/2$ by: $\alpha \to (\alpha, id)$ if $\alpha \in E$ and $\alpha \to (\alpha v_n, \tau)$ if $\alpha $ is odd, where $v_n$ is the last vertex of the walk $\alpha$.  Thus if $\alpha$ is odd, we repeat the last vertex (which we can do since all vertices are looped) to create an even walk.  

The map $\Psi$ is an isomorphism on vertices, since the vertices of $E$ are the same as the vertices of $\Pi$. We check that it is a functor.    If $\alpha$ is even then it is easy to see that $\Psi(\alpha \beta) = \Psi(\alpha)\Psi(\beta)$.  If $\alpha$ is odd and $\beta$ is even we need to compare $\alpha v_n \beta$ with $\alpha \beta w_n$. But  these are homotopic rel enpoints, since all vertices are looped and so  we have a sequence of spider moves that move the repeated vertex down through $\beta$ to the end.  Similarly, if $\alpha$ and $\beta$ are both odd we are comparing $\alpha v_n \beta w_n$ to $\alpha \beta$;  again we have a sequence of spider moves that take the repeated vertex  to the end to get $\alpha \beta w_n w_n$ which prunes to $\alpha \beta$. 

We define an inverse map  $\Lambda(\alpha, id) = \alpha$ and $\Lambda (\alpha, \tau) = \alpha v_n$.    Then $\Lambda\Psi$ and $\Psi\Lambda$ are identities since on evens they are identities and on odds they send $\alpha$ to $\alpha v_n v_n$ which prunes to $\alpha$, showing that $\Psi$ is an isomorphism.

\end{proof}

Now we look at the pushout graph and prove a modified van Kampen theorem \cite{Hatcher, VanKampen}.  

\begin{theorem}
       If $G = G_1 \cup G_2$ and  all diamonds (induced cycles of length 4) of $G$ are fully contained in either $G_1$ or $G_2$ then $\Pi(G) = \Pi(G_1) *_{\Pi(G_1 \cap G_2)} \Pi(G_2)$.  
\end{theorem}

\begin{proof}
We verify that $\Pi(G)$ has the universal property for a pushout diagram:  suppose we have two groupoid maps $\varphi_1, \varphi_2:  \Pi(G_i) \to R$ for some groupoid $R$.  Then we can define a map $\varphi:  \Pi(G) \to R$ as follows.  For any arrows $\alpha = (v w_1 w_2 \dots w)$, we can break it up into pieces $\alpha = \alpha_1 * \alpha_2 \* \alpha_3 \dots$ where each piece is contained in either $G_1$ or $G_2$.  Then we define $\varphi(\alpha) = \varphi_i(\alpha_k)$ where we apply the map $\varphi_1$ to pieces in $G_1$ and $\varphi_2$ to pieces in $G_2$.  This is well-defined, since if any piece is in both $G_1$ and $G_2$ then $\phi_1 = \phi_2$, and any spider move will take place in either $G_1$ or $G_2$ by our diamond condition.  It is unique since the functor $\varphi$ needs to agree with $\varphi_1$ and $\varphi_2$ and respect the concatenation operation.   Thus $\Pi(G)$ is the groupoid pushout.  
\end{proof}


\section{Comparison with Other Fundamental Groups for Graphs} \label{S:Comp}

There is another fundamental group which has been defined based on $\times$-homotopy by  \cite{Docht2}.   This corresponds to a looped version of our fundamental groupoid which we sketch here.

Our fundamental groupoid $\Pi(G)$ is based on homotopy classes  walks defined by $P_n \to G$.  It is also possible to define a looped fundamental groupoid based on homotopy classes of walks $I_n^\ell \to G$, so that all the vertices in the objects and in any walk need to be looped.  This will effectively be an invariant of the looped subgraph of a graph $G$.

\begin{definition} Let $\alpha = (v_0 v_1 v_2 \dots  v_n)$ be a looped walk in $ G$.    We say that $\alpha$ is {\bf $\ell$-prunable} if it is prunable or if $v_i = v_{i+1}$ for some $i$.   We  define a {\bf $\ell$-prune} of $\alpha$ either to be a prune or to be given by a walk  $\alpha'$ obtained by deleting one of the repeated vertices $v_i$ from the walk when $v_i = v_{i+1}$:     
if 
$$ \alpha = (v_0 v_1 v_2 \dots v_{i-1} v_{i} v_{i} v_{i+2} v_{i+3}  \dots v_n) $$  
then the $\ell$-prune  of $\alpha$ is  $$ \alpha' = (v_0 v_1 v_2 \dots v_{i-1}  v_i v_{i+2}  \dots v_n) $$  

\end{definition}

Then we can make the following definition.  
\begin{definition} \label{D:lg} Let $\Pi^{\ell}(G)$ be the looped fundamental groupoid of $G$  defined by the following:  

\begin{itemize}
    \item objects of $\Pi^{\ell}(G)$ are vertices of the graph $G$
    \item an arrow from $v_0$ to $v_n$ in $\Pi^\ell (G)$ is given by a $\ell$-prune class of walks from $v_0$ to $v_n$ defined up to  homotopy rel endpoints
    \item composition of arrows is defined using concatenation of walks 
\end{itemize}
\end{definition}

Verifying that this is a well-defined groupoid and that $\Pi^\ell$ defines a homotopy invariant for finite graphs is a straightforward adaptation of the arguments given in Section \ref{S:fgp} for the unlooped version.   However, the looped groupoid NOT equivalent to the   unlooped even if all vertices are looped, since the requirement for a homotopy of $I_n^\ell$ is stricter than that for $P_n$ and any spider move must swap images between {\bf connected} vertices.     This is illustrated in the example below.  

\begin{example}\label{Example:LoopedSquare}
Consider $G$ depicted below:

$$
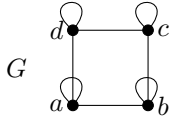
\begin{tikzpicture}
\draw (0,0)--(1,0)--(1,1)--(0,1)--(0,0);
\draw[fill] (0,0) circle (2pt);
\draw[fill] (1,0) circle (2pt);
\draw[fill] (0,1) circle (2pt);
\draw[fill] (1,1) circle (2pt);

\draw (0,0)  to[in=50,out=140,loop, distance=.7cm] (0,0);
\draw (0,1)  to[in=50,out=140,loop, distance=.7cm] (0,1);
\draw (1,0)  to[in=50,out=140,loop, distance=.7cm] (1,0);
\draw (1,1)  to[in=50,out=140,loop, distance=.7cm] (1,1);

\draw (0,0) -- node[left]{$a$} (0,0);
\draw (0,1) -- node[left]{$d$} (0,1);
\draw (1,0) -- node[right]{$b$} (1,0);
\draw (1,1) -- node[right]{$c$} (1,1);

\node at (-.75,.5){$G$};

\end{tikzpicture}$$

Consider the walk $((abc)$ from $a$ to $c$.   In $\Pi(G)$, this walk is homotopic to $(adc)$  via a spider-move from $b$ to $d$. 
However in $\Pi^{\ell}(G)$  $(abc)\neq(adc)$, since there is no homotopy from $I_3^\ell$ taking  $b$ to $d$:  since the vertices of $I_3^\ell$ are looped, such a spider move would require an edge from  $b$ and $d$.

\end{example}

We can think of  looped walks as  infinite length walks which stabilize at some point, so for some $n$, then for all $m\geq n$ the walk is the same vertex $v_n$.   Thus we see that it generalizes the definition of the fundamental group given by \cite{Docht2} for  exponential objects $H^G$:  if we pick a basepoint then the group of homotopy classes of looped walks $\Pi^\ell(H^G, v)$  is isomorphic to  Dochtermann's fundamental group  defined by $[1_*, \Omega(H^G)]_{\times}$.     Our definition applies to any graph,  not just an exponential one. 

We can also generalize  \cite{Docht2}  Corollary 4.8 giving a connection to the   polyhedral hom complex using the same approximation techniques.

\begin{definition} \label{D:homcx} \cite{CompGH}
The polyhedral complex  $ \Delta = \Hom(G, H)$ has cells indexed by 
functions $\eta:   V (G) \to  2^{V (H)} \backslash \{\emptyset\}$, such that if $x\con y \in E(G)$, then $\eta(x) \times \eta(y) \subseteq
E(H)$.
The boundary attachemnts of the polyhedrons are defined by inclusions $\eta \subseteq \eta'$.  
\end{definition}

The 2-skeleton of this complex is described explicitly by: 
\begin{itemize}
    \item $0$-cells are indexed by  graph homomorphism $G \to H$.
    \item $1$-cells will have a single vertex $v$ such that $|\eta(v)| = 2$.  Then $\eta$ defines a 1-cell connecting the two 0-cells indexed by the morphisms defined by the two choices of image of $v$, and these two  are connected by a spider move.   
    \item   $2$-cells are of two types:  
    A single vertex $v$ with $|\eta(v)| = 3$, giving a 2-cell  filling in a  triangle of shape (A), or   two vertices $v, w$ with $|\eta(v)| = |\eta(w)| = 2$,  giving a 2-cell  filling in square of shape (B):  \\     (A) \phantom{www}  $
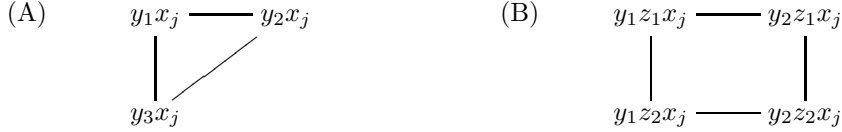
\xymatrix{ y_1x_j \ar@{-}[r] \ar@{-}[d] & y_2x_j \ar@{-}[dl] \\y_3x_j }$ \phantom{WWWWWW}
        (B) \phantom{www}$ \xymatrix{y_1z_1 x_j\ar@{-}[r] \ar@{-}[d] & y_2z_1x_j \ar@{-}[d] \\  y_1z_2x_j \ar@{-}[r] & y_2z_2 x_j }$
\end{itemize}
 
 We will use this to show that the fundamental groupoid of $\Hom(G, H)$ is equivalent to the looped groupoid of the exponential graph $H^G$.  To do this, we need the following lemma.  

\begin{lemma} \label{L:key}  Given any four morphisms $v, x, \hat{x}, w$ such that $[(v x w)] = [(v \hat{x} w)]$  in $\Pi^\ell(G)$   we can fill the interior of the diamond in with triangles and squares of the form (A) and (B).  

\end{lemma}
\begin{proof}

We  induct on the total number of vertices which have different images under one or more pair of the morphisms  $v, x, \hat{x}, w$. 
If $k = 1$  then all of these morphisms agree on everything but a single vertex, and we can fill in with triangles of form (A).

Now suppose the images of  $n$ vertices differing.  We will choose an  ordering for these vertices $a_{1},\ldots,a_n$, and assume that all morphisms to be discussed will agree on any other vertex of $G$.  
By Proposition \ref{P:spider}, we have a sequence of $n$ spider moves from $v$ to $ x$,   consisting of $v_j$  where each $v_j$ agrees with $v$ on $a_k$ for $k \geq j$, and agrees with $x$ on $a_k $ for $k < j$.  Thus as we work through the $v_j$, we move the images of $a_k$ in increasing order.   
Similarly, we have spider moves from $w$ to $x$, consisting of  $w_j$ changing the images of the vertices in order from $w$ to $x$; and $\hat{v}_j$ from $v$ to $\hat{x}$, and lastly $\hat{w}_j$ from $w$ to $\hat{x}$.  
We then fill out the diamond with morphisms $u_k, y_k$ as follows:

$$
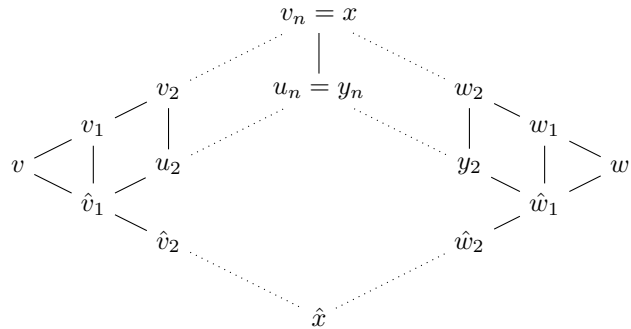
\begin{tikzpicture}
\node (vi-1) at (0,0){$v$};
\node (fi1) at (1,1/2){$v_1$};
\node (fi2) at (2,2/2){$v_2$};
\node (vi) at (4,4/2){$v_n = x$};
\node (fi+12) at (6,2/2){$w_2$};
\node (fi+11) at (7,1/2){$w_1$};
\node (vi+1) at (8,0){$w$};

\node (f'i1) at (1,-1/2){$\hat{v}_1$};
\node (f'i2) at (2,-2/2){$\hat{v}_{2}$};
\node (v'i) at (4,-4/2){$\hat{x}$};
\node (f'i+12) at (6,-2/2){$\hat{w}_2$};
\node (f'i+11) at (7,-1/2){$\hat{w}_1$};

\node (gi2) at (2,0/2){$u_2$};
\node (gin) at (4,2/2){$u_n = y_n$};
\node (gi+12) at (6,0/2){$y_2$};

\draw (vi-1)--(fi1)--(fi2);
\draw[dotted] (fi2)--(vi)--(fi+12);
\draw (fi+12)--(fi+11)--(vi+1);

\draw (vi-1)--(f'i1)--(f'i2);
\draw[dotted] (f'i2)--(v'i)--(f'i+12);
\draw (f'i+12)--(f'i+11)--(vi+1);
\draw (fi1)--(f'i1);
\draw (fi+11) -- (f'i+11);
\draw (f'i1)--(gi2);
\draw[dotted] (gi2)--(gin)--(gi+12);
\draw (gi+12)--(f'i+11);
\draw (fi2)--(gi2);
\draw (vi)--(gin);
\draw (fi+12)--(gi+12);

\end{tikzpicture}$$

where  $u_j$ is defined to aree with $v_j$ on all vertices except for $a_1$, and take the same value as $\hat{x}$ on $a_1$;  and similarly $y_j$ agrees with $w_j$ on all vertices except for $a_1$, and with $\hat{x}$ on $a_1$.  Then  the  bars in the diagram above are spider pairs because they only differ on a single vertex ($a_1$ for the vertical bars, and successive $a_k$ for the diagonals), and the squares and triangles between the top and second lines are of the form (A) and (B).

Thus what remains is paths $(\hat{v}_1 u_n \hat{w}_1)$ and  $(\hat{v}_1 \hat{x} \hat{w}_1)$ who all agree on $a_1$, and thus disagree on $n-1$ vertices.  Our inductive hypothesis fills in the interior of this interior diamond.

\end{proof}

 \begin{theorem} \label{T:equiv}
    Let $K = H^G $ be  the  exponential graph, and let $\Delta=\Hom(G, H)$ of Definition \ref{D:homcx}.      There is an equivalence of categories $\Pi^{\ell}(K) \simeq \Pi(\Delta)$ where  $\Pi^{\ell}(K)$  is the looped groupoid from Definnition \ref{D:lg},  and  $\Pi(\Delta)$ is the topological fundamental groupoid of the space $\Delta$.  
\end{theorem}

\begin{proof}

Define $\Phi:  \Pi^{\ell}(K)\to \Pi(\Delta)$  as follows:  if  $v$ is an object of $\Pi^\ell (K)$ then is a looped vertex of $K = H^G$ which defines a morphism $G \to H$ which corresponds to a $0$-cell.  Send the object $v$ to the object represented by this $0$-cell in $\Pi(\Delta)$.   If 
 $\alpha = (v_0 v_1 v_2 \dots v_n) $ represent an arrow of $\Pi^{\ell}(K)$, then  $v_i \con v_{i+1}$ in $H^G$, and so we have a sequence of  spider moves $v_i f_1 f_2 \dots f_m v_{i+1}$ connecting the morphisms $v_i $ and $v_{i+1}$, each connecting   morphisms which differ in the image of a single vertex $v$,  and thus corresponding to a 1-cell of $\Delta$.     Send $\alpha$ to the path along the 1-cells.  This is independent of choice of spider move, since a different choice would correspond to a different order of moving the vertices one at a time, and we can fill in two such choices with a square of type (B) from the 2-skeleton.  Thus two choices of spider realizations are homotopic in $\Delta$.  
 
Now if  $[\alpha] = [\beta]$ in $ \Pi^{\ell}(K)$, then they are homotopic rel endpoints up to $\ell$-pruning.  A prune comes from a repeated vertex, which would be mapped under $\Phi$ to a path in $\Delta$ which was constant at that vertex, homotopic to the walk without the pause.   And any homotopy rel endpoints could be realized by a sequence of spider moves which could be filled in by  Lemma \ref{L:key}. 

To show that  $\Phi$ is essentially surjective on objects, let  $x \in \Pi(\Delta)$ be an object of the fundamental groupoid and hence a  point in $\Delta$.  Choose any corner $y$  of its simplex and a path  $\gamma$ from $x$ to $y$.  Then $y$ is in the image of $\Phi$ and $\gamma$ represents an arrow from $x$ to $y$.   

To show that $\Phi$ is full on arrows,  suppose that  there is a path in $\Delta$ from 0-cell $v$ to $w$.  Then   $\gamma$ is homotopic to $\gamma'$ that lies in the 1-skeleton of $\Delta$ by cellular approximation \cite{Hatcher}, and $\gamma'$ is in the image of $\Phi$.   to show that $\Phi$ is faithful  on arrows,  
suppose that  $\alpha, \beta:  v \to w$ in $\Pi(K)$ given by paths $\alpha = (v v_1 v_2 \dots w)$ and $\beta = (v w_1 w_2 \dots w)$, such that  $\Phi(\alpha) = \Phi(\beta)$ in $\Pi(\Delta)$.   This means that there is a homotopy from $\alpha$ to $\beta$ in $\Delta$ which we may assume lives in the 2-skeleton, so lives on triangles of type (A) and squares of type (B).  Each of these corresponds to spider moves showing that $[\alpha] = [\beta] $ in $\Pi^\ell(K)$.  
\end{proof}

\section*{Acknowledgements}

The authors wish to thank Dr.\ Anton Dochtermann for his clarification, guidance and encouraging words.  We also wish to thank everyone of the Talk Math With Your Friends $\#$TMYWF community for their questions which led to some of these results.

\bibliographystyle{spmpsci}      
\bibliography{ref}   

\end{document}